\documentclass{amsart}
\usepackage{amssymb}
\usepackage[all]{xy}

\newcommand{\IR}{\mathbb R}

\newcommand{\IN}{\mathbb N}
\newcommand{\w}{\omega}
\newcommand{\U}{\mathcal U}

\newcommand{\V}{\mathcal V}
\newcommand{\C}{\mathcal C}
\newcommand{\A}{\mathcal A}

\newcommand{\Ra}{\Rightarrow}

\newcommand{\e}{\varepsilon}
\newcommand{\diam}{\mathrm{diam}}

\newcommand{\IS}{\mathbb S}

\newtheorem{theorem}{Theorem}[section]
\newtheorem{corollary}[theorem]{Corollary}
\newtheorem{proposition}[theorem]{Proposition}
\newtheorem{problem}[theorem]{Problem}
\newtheorem{lemma}[theorem]{Lemma}
\newtheorem{example}[theorem]{Example}
\theoremstyle{definition}
\newtheorem{definition}[theorem]{Definition}

\title[Piotrowski spaces]{Quasicontinuous functions with values in Piotrowski spaces}
\author{Taras Banakh}
\address{Ivan Franko National University of Lviv (Ukraine) and Jan Kochanowski University in Kielce (Poland)}
\email{t.o.banakh@gmail.com}
\thanks{The author has been partially financed by NCN grant DEC-2012/07/D/ST1/02087.}
\dedicatory{Dedicated to the memory of Zbigniew Piotrowski (1953--2014)}

\begin{document}

\begin{abstract} A topological space $X$ is called {\em Piotrowski} if every quasicontinuous map $f:Z\to X$ from a Baire space $Z$ to $X$ has a continuity point. In this paper we survey known results on Piotrowski spaces and investigate the relation of Piotrowski spaces to strictly fragmentable, Stegall, and game determined  spaces. Also we prove that a Piotrowski Tychonoff space $X$ contains a dense (completely) metrizable Baire subspace if and only if  $X$ is Baire (Choquet).
\end{abstract}
\maketitle

\section{Introduction}

The problem of evaluation of the sets of discontinuity points of quasicontinuous maps is classical in Real Analysis and traces its history back to Volterra, Baire and Kempisty.

Let us recall that a function $f:X\to Y$ between topological spaces is {\em quasicontinuous at a point} $x\in X$ if for any neighborhood $O_x\subset X$ of $x$ and any neighborhood $O_{f(x)}\subset Y$ of $f(x)$ there exists a non-empty open set $U\subset O_x$ such that $f(U)\subset O_{f(x)}$. A function $f:X\to Y$ is {\em quasicontinuous} if it is quasicontinuous at each point $x\in X$. Formally, quasicontinuous functions were introduced by Kempisty \cite{Kemp} but implicitly they appeared earlier in works of Baire and Volterra.

It is well-known (see \cite{Bled}, \cite{Lev63}, \cite[3.1.1]{Neu}, \cite{Piot85}, \cite{Piot90}) that every quasicontinuous function $f:X\to Y$ from a non-empty Baire space $X$ to a metrizable space $Y$ has a continuity point. More precisely, the set $C(f)$ of continuity points of $f$ is dense $G_\delta$ in $X$.  We recall that a topological space $X$ is {\em Baire} if for any sequence $(U_n)_{n\in\w}$ of open dense subsets in $X$ the intersection $\bigcap_{n\in\w}U_n$ is dense in $X$.

In \cite{Piot98} Piotrowski suggested to study spaces $Y$ for which every quasicontinuous function $f:X\to Y$ defined on a non-empty Baire space $X$ has a continuity point. Honoring the contribution of Zbigniew Piotrowski to studying quasicontinuous maps, we suggest to call such spaces $Y$ {\em Piotrowski} spaces. The mentioned problem of Piotrowski was considered in the seminal paper \cite{KKM} of Kenderov, Kortezov and Moors, containing many interesting and non-trivial results obtained with the help of topological games. It turns out that Piotrowski spaces are tightly connected with fragmentable and Stegall spaces, well-known spaces in  General Topology and its applications to Functional Analysis, see \cite{Fab}, \cite{Gil82}, \cite{Ph89}, \cite{St91}.

 In this paper we survey known results on Piotrowski spaces and also prove some new results. In particular, in Section~\ref{s:DM} we shall prove that a Piotrowski Tychonoff space $X$ contains a dense Baire (completely) metrizable subspace if and only $X$ is Baire (Choquet). In Section~\ref{s:stabP} we establish some stability properties of the class of Piotrowski spaces.
\smallskip

{\bf Acknowledgement.} The author would like to thank the anonymous referee of an initial version of this paper for valuable remarks, suggestions, and references, which resulted in splitting the initial paper into two parts. This paper is the first part of the splitted paper; the second part \cite{B-M} is devoted to Maslyuchenko spaces and uses the results of this paper (devoted to Piotrowski spaces).

  \section{$\C$-Piotrowski and strongly $\C$-Piotrowski spaces}

It is be convenient to insert a parameter into the definition of a Piotrowski space and define a more general notion.

 \begin{definition}\label{d:piotr} Let $\C$ be a class of Baire spaces. A topological space $X$ is called
\begin{itemize}
\item {\em $\C$-Piotrowski} if every quasicontinuous map $f:C\to X$ defined on a non-empty space $C\in\C$ has a continuity point;
\item {\em strong $\C$-Piotrowski} if for every quasicontinuous map $f:C\to X$ defined on a space $C\in\C$ the set $C(f)$ of continuity points of $f$ is comeager in $C$.
\item ({\em strong}) {\em Piotrowski} if $X$ is (strong) $\C$-Piotrowski for the class $\C$ of all Baire topological spaces.
\end{itemize}
\end{definition}

We recall that a subset $M$ of a topological space $X$ is called {\em meager} if $M$ can be written as the countable union of nowhere dense subsets of $X$. A subset $C\subset X$ is called {\em comeager} in $X$ if its complement $X\setminus C$ is meager in $X$.

It is clear that each strong $\C$-Piotrowski space is $\C$-Piotrowski.
If the class $\C$ is closed under taking open subspaces and dense Baire subspaces, then for regular topological spaces the converse implication is also true.

We shall say that a class $\C$ of topological spaces is closed under taking
\begin{itemize}
\item {\em open subspaces} if for every space $C\in\C$ each open subspace of $C$ belongs to the class $\C$;
\item {\em dense $G_\delta$-subsets} if for every space $C\in\C$ each dense $G_\delta$-subset of $C$ belongs to the class $\C$;
\item {\em dense Baire subspaces} if for every space $C\in\C$ each dense Baire subspace of $C$ belongs to the class $\C$.
\end{itemize}

We shall need to following (probably) known heredity property of quasicontinuous maps.

\begin{lemma}\label{l:Prest} Let $f:X\to Y$ be a quasicontinuous map from a topological space $X$ to a regular topological space $Y$. For every dense subset $Z\subset X$ we get $C(f|Z)=C(f)\cap Z$.
\end{lemma}

\begin{proof} Given a continuity point $z\in Z$ of the restriction $f|Z$, we should prove that $z$ remains a continuity points of the function $f$.
Given any neighborhood $O_{f(z)}\subset X$ we should find a neighborhood $U_z\subset Z$ of $z$ such that $f(U_z)\subset O_{f(z)}$. By the regularity of $X$, the point $f(z)$ has an open neighborhood $W\subset X$ such that $\overline{W}\subset O_{f(z)}$.

By the continuity of the restriction $f|Z$ at $z$, there exists a neighborhood $V_z\subset X$ of $z$ such that $f(V_z\cap Z)\subset W$. We claim that $f(V_z)\subset \overline{W}\subset O_{f(z)}$. To derive a contradiction, assume that $f(v)\notin\overline{W}$ for some point $v\in V_z$. By the quaiscontinuity of $f$, there exists a non-empty open set $V'\subset V_z$ such that $f(V')\subset X\setminus\overline{W}$. Since $V_z\cap Z$ is dense in $V_z$, there exists a point $v\in V'\cap Z$. For this point we get $f(v)\in f(V_z\cap Z)\cap f(V_z)\subset W\cap(X\setminus\overline{W})=\emptyset$, which is a desired contradiction. This contradiction shows that $f(V_z)\subset \overline{W}\subset O_{f(z)}$ and $f$ is continuous at $z$.
\end{proof}

The following proposition is a modification of Proposition 1 of \cite{K99} and can be easily derived from Lemma~\ref{l:Prest}.

\begin{proposition}\label{p:Pgood} Let $\C$ be a class of Baire spaces and $f:C\to X$ be a quasicontinuous map from a non-empty space $C\in\C$ to a $\C$-Piotrowski regular space $X$.
\begin{enumerate}
\item If the class $\C$ is closed under open subspaces, then the set $C(f)$ of continuity points of $f$ is dense in $X$.
\item If $\C$ is closed under dense $G_\delta$-subspaces, then the set $C(f)$ is not meager in $X$.
\item If $\C$ is closed under open subspaces and dense $G_\delta$-subspaces, then the set $C(f)$ is a dense Baire subspace of $X$.
\item If $\C$ is closed under open subspaces and dense Baire subspaces, then the set $C(f)$ is comeager in $X$.
\end{enumerate}
\end{proposition}

The last statement of this proposition implies the following characterization.

\begin{corollary}\label{c:P<=>sP} If a class $\C$ of Baire spaces is closed under open subspaces and dense Baire subspaces, then a (regular) topological space $X$ is $\C$-Piotrowski if (and only if) $X$ is strong $\C$-Piotrowski.
\end{corollary}

\begin{corollary} A (regular) topological space $X$ is Piotrowski if (and only if) $X$ is strong Piotrowski.
\end{corollary}

\section{$\C$-Stegall and strong $\C$-Stegall spaces}

Piotrowski spaces are tightly related to Stegall spaces, introduced by Stegall \cite{St83}, \cite{St91} and extensively studied in Functional Analysis, see \cite{Fab}, \cite{HH}, \cite{Gil82}, \cite{Ph89} and references therein. Stegall spaces are defined with the help of minimal usco maps.

A multi-valued map $\Phi:X\multimap Y$ between topological spaces will be called an {\em usco map} if for every $x\in X$ the set $\Phi(x)\subset Y$ is non-empty compact, and $\Phi$ is {\em upper semicontinuous} in the sense that for every open set $U\subset Y$ the set $\{x\in X:\Phi(x)\subset U\}$ is open in $X$. An usco map $\Phi:X\multimap Y$ is called {\em minimal} if $\Phi=\Psi$ for any usco map $\Psi:X\multimap Y$ with $\Psi(x)\subset\Phi(x)$ for all $x\in X$.

For a multivalued map $\Phi:Z\multimap X$ and a subset $V\subset Z$ put $\Phi(V):=\bigcup_{v\in V}\Phi(v)\subset X$.

The following theorem yields an important example of a minimal usco map. In this theorem for a compact space $K$ by $C(K)$ we denote the Banach space of all continuous real-valued functions on $K$, endowed with the sup-norm $\|f\|=\sup_{x\in K}|f(x)|$.

\begin{lemma}\label{l:CK-musco} For every compact Hausdorff space $K$ the multi-valued map $$\Phi:C(K)\multimap K,\;\;\Phi(f):=\{x\in K:f(x)=\max f(K)\},$$
is minimal usco. Moreover, for every open set $U\subset C(K)$ the set $\Phi(U)$ is open in $K$.
\end{lemma}

The proof of this lemma can be found in Lemma 2.2.1 and Theorem 3.1.6 of \cite{Fab}.

We shall often use the following (known) characterization of minimal usco maps (see \cite[3.1.2]{Fab}).

\begin{lemma}\label{l:musco} An usco map $\Phi:Z\multimap X$ from a topological space $Z$ to a (Hausdorff) topological space $X$ is   minimal usco (if and) only if for any open sets $V\subset Z$, $W\subset X$ with $W\cap\Phi(V)\ne\emptyset$ there exists a non-empty open set $U\subset V$ such that $\Phi(U)\subset W$.
\end{lemma}

The ``only if'' part of this characterization implies the following lemma connecting quasicontinuous and minimal usco maps (cf. \cite[Corollary 2]{KKM} and \cite[3.4]{HH}).

\begin{lemma}\label{l:select} For any minimal usco map $\Phi:Z\multimap X$ between  topological  spaces, every selection $\varphi:Z\to X$ of $\Phi$ is quasicontinuous.
\end{lemma}

A map $\varphi:Z\to X$ is called a {\em selection} of a multivalued map $\Phi:Z\multimap X$ if $\varphi(z)\in\Phi(z)$ for all $z\in Z$.
\smallskip

Lemma~\ref{l:musco} implies also the following useful fact:

\begin{lemma}\label{l:musco-rest} Let $\Phi:Z\multimap X$ be a minimal usco map from a topological space $Z$ to a (Hausdorff) topological space $X$, and $Y$ be an open (or dense) subspace of $Z$. Then the restriction $\Phi|Y:Y\multimap X$ is a minimal usco map.
\end{lemma}

\begin{definition}\label{d:Stegall} Let $\C$ be a class of Baire spaces.
A topological space $X$ is defined to be
\begin{itemize}
\item {\em $\C$-Stegall} if every minimal usco map $\Phi:C\multimap X$ is single-valued at some point $z\in C$;
\item {\em strong $\C$-Stegall} if every minimal usco map $\Phi:C\multimap X$ is single-valued at all points of some comeager subset of $C$;
\item ({\em strong}) {\em Stegall} if $X$ is (strong) $\C$-Stegall for the class $\C$ of all Baire topological spaces.
\end{itemize}
\end{definition}

Stegall spaces have found many applications in the geometry of Banach spaces \cite{Fab}, \cite{KM04}, \cite{NP}, \cite{St87}, differentiability theory of convex and Lipschitz functions \cite{BM}, \cite{Fab}, \cite{Gil82}, \cite{Ph89}, \cite{St83}, optimization \cite{CK}, \cite{CKR}.  In \cite{K99} $\C$-Stegall spaces are called Stegall spaces with respect to the class $\C$.

The following proposition was proved in \cite[Proposition 1]{K99} and is a counterpart of Proposition~\ref{p:Pgood}. This proposition can be easily derived from Lemma~\ref{l:musco-rest} and Definition~\ref{d:Stegall}.

\begin{proposition}\label{p:Sgood} Let $\C$ be a class of Baire spaces and $\Phi:C\to X$ be a minimal usco map from a non-empty space $C\in\C$ to a $\C$-Stegall Hausdorff space $X$.
\begin{enumerate}
\item If the class $\C$ is closed under open subspaces, then $\Phi$ is single-valued at all points of some dense subset of $C$.
\item If $\C$ is closed under dense $G_\delta$-subspaces, then
$\Phi$ is single-valued at all points of some non-meager subset of $C$.
\item If $\C$ is closed under open subspaces and dense $G_\delta$-subspaces, then
$\Phi$ is single-valued at all points of some dense Baire subspace of $C$.
\item If $\C$ is closed under open subspaces and dense Baire subspaces, then
$\Phi$ is single-valued at all points of some comeager subset of $C$.
\end{enumerate}
\end{proposition}

The last statement of this proposition implies:

\begin{corollary} A topological space $X$ is Stegall if and only if $X$ is strong Stegall.
\end{corollary}

We are going to prove that each (strong) $\C$-Piotrowski Hausdorff space is (strong) $\C$-Stegall.

\begin{theorem}\label{t:P->S} Let $\C$ be a class of Baire topological spaces. Every  (strong) $\C$-Piotrowski Hausdorff space $X$ is (strong) $\C$-Stegall.
\end{theorem}

\begin{proof} Given a minimal usco map $\Phi:C\multimap X$ from a non-empty space $C\in\C$, we need to prove that $\Phi$ is single-valued at all points of some non-empty (comeager)  subspace of $C$. By Lemma~\ref{l:select}, any selection  $\varphi:C\to X$ of the set-valued map $\Phi$ is quasicontinuous. Since the space $X$ is  (strong) $\C$-Piotrowski, the function $\varphi$ is continuous at points of some non-empty  (comeager) set $Z\subset C$. We claim that $\Phi$ is single-valued at each point $z\in Z$. Assuming that $|\Phi(z)|>1$, we can find a point $y\in\Phi(z)\setminus\{\varphi(z)\}$ and choose an open neighborhood $W\subset X$ of $y$ such that $\varphi(z)\notin\overline{W}$. The continuity of $\varphi$ at the point $z$ yields an open neighborhood $U_z\subset Z$ of $z$ such that $\varphi(U_z)\subset X\setminus\overline{W}$. It can be shown that the multi-valued map
map $\Phi':C\multimap X$ defined by
$$\Phi'(z')=\begin{cases}
\Phi(z')\setminus W&\mbox{if $z'\in U_z$},\\
\Phi(z'),&\mbox{otherwise},
\end{cases}
$$
is usco, which contradicts the minimality of the usco map $\Phi$.
\end{proof}

 \section{Fragmentable and strictly fragmentable spaces}

 A topological space $Y$ is {\em fragmented\/} by a metric $d\/$ if for every $\e>0$, each non-empty subspace $A\subset X$ contains a non-empty relatively open subset $U\subset A$ of $d$-diameter $\diam(U)<\e$. If the metric $d$ generates a topology at least as strong as the original topology of $X$, then we shall say that $X$ is {\em strictly fragmented} by the metric $d$. A topological space is called ({\em strictly}) {\em fragmentable} if it is (strictly) fragmented by some metric. It is clear that each metrizable space is strictly fragmentable and each strictly fragmentable space is fragmentable. By \cite{Rib}, a compact Hausdorff space is strictly fragmentable if and only if it is fragmentable.  Each strictly fragmentable space is strongly fragmentable in the sense of Reznichenko \cite[p.104]{Ar96}.

The following theorem was proved by Ribarska \cite{Rib} (see also \cite[5.1.11]{Fab}).

\begin{theorem}[Ribarska]\label{t:F->S} Each fragmentable space $X$ is strong Stegall.
\end{theorem}

 \begin{proof} We present a short proof of this theorem, for convenience of the reader. Let $d$ be a metric that fragments $X$ and $\Phi:Z\multimap X$ be a minimal usco map defined on a non-empty Baire topological space $Z$. For every $\e>0$ denote by $\U_\w$ the family of all open subsets of $U\subset Z$ such that the image $\Phi(U)=\bigcup_{z\in U}\Phi(z)$ has $d$-diameter $<\e$. We claim that $\bigcup\U_\e$ is dense in $Z$. Indeed, take any non-empty open set $W\subset Z$. By the fragmentability of $X$, the set $\Phi(W)$ contains a non-empty relatively open subset $V\subset \Phi(W)$ of $d$-diameter $<\e$. By Lemma~\ref{l:musco}, there exists a non-empty open set $U\subset W$ such that $\Phi(U)\subset V$ and hence $U\in\U_\e$.  So, $\bigcup\U_\e$ is dense in $X$ and $G=\bigcap_{n\in\w}\U_{2^{-n}}$ is a dense $G_\delta$-set in $X$. It is clear that for every $z\in Z$ the $d$-diameter of the set $\Phi(z)$ is zero, which means that $\Phi(z)$ is a singleton.
 \end{proof}

A corresponding result for strictly fragmentable spaces was obtained in \cite[Theorem 2.7]{GKMS}, \cite[Lemma 4.1]{MG}, \cite[Theorem 5.1]{KM3}, \cite[Theorem 1]{KKM}.

\begin{theorem}[Giles-Kenderov-Kortezov-Moors-Sciffer]\label{t:sF-P} Each strictly fragmentable space $X$ is strong Piotrowski.
\end{theorem}

 \begin{proof} We present a short proof of this theorem, for convenience of the reader. Let $d$ be a metric that strictly fragments $X$ and $\varphi:Z\to X$ be a quasicontinuous map defined on a non-empty Baire topological space $Z$. For every $\e>0$ denote by $\U_\w$ the family of all open subsets of $U\subset Z$ such that the image $\varphi(U)$ has $d$-diameter $<\e$. We claim that $\bigcup\U_\e$ is dense in $Z$. Indeed, take any non-empty open set $W\subset Z$. By the choice of the metric $d$, the set $\varphi(U)$ contains a non-empty relatively open subset $V\subset \varphi(U)$ of $d$-diameter $<\e$. By the quasicontinuity of the map $\varphi$, there exists a non-empty open set $U\subset W$ such that $\varphi(U)\subset V$ and hence $U\in\U_\e$.  So, $\bigcup\U_\e$ is dense in $X$ and $G=\bigcap_{n\in\w}\U_{2^{-n}}$ is a dense $G_\delta$-set in $X$. It is clear that each point $z\in G$ is a continuity point of the function $\varphi:Z\to X$ (as the topology generated by the metric $d$ contains the topology of the space $X$).
 \end{proof}

Theorems~\ref{t:P->S}, \ref{t:F->S}, \ref{t:sF-P} imply that for a Hausdorff  topological space and any class $\C$ of Baire spaces we have the following implications:
$$
\xymatrix{
\mbox{strictly fragmentable}\ar@{=>}[r]\ar@{=>}[d]&\mbox{Piotrowski}\ar@{=>}[r]\ar@{=>}[d]&\mbox{strong $\C$-Piotrowski}\ar@{=>}[d]\ar@{=>}[r]&\mbox{$\C$-Piotrowski}\ar@{=>}[d]\\
 \mbox{fragmentable}\ar@{=>}[r]&\mbox{Stegall}\ar@{=>}[r]&\mbox{strong $\C$-Stegall}\ar@{=>}[r]&\mbox{$\C$-Stegall}.
}
 $$
The vertical implications in this diagram can be reversed for game determined spaces introduced by Kenderov, Kortezov and Moors \cite{KKM} and discussed in the next section.

\section{Game characterizations}

In this section we present game characterizations of (strictly) fragmentable spaces due to Kenderov, Kortezov and Moors \cite{KKM}. These characterizations involve the (strict) fragmenting game, which we are going to describe now.

The (strict) fragmenting game involves two players $\Sigma$ and $\Omega$. Given a topological space $X$, the players select, one after the other, non-empty subsets of $X$. The player $\Omega$ starts the game by selecting the whole space $X$ and $\Sigma$ answers  by choosing any subset $A_1$ of $X$ and $\Omega$ goes on by taking a subset $B_1\subset A_1$ which is {\em relatively open} in $A_1$. After that, on the $n$th stage of development of the game, $\Sigma$ takes any subset $A_n$ of the last move $B_{n-1}$ of $\Omega$ and the latter answers by taking again a relatively open subset $B_n$ of the set $A_n$ just chosen by $\Sigma$. Acting this way the players produce a sequence of non-empty sets $A_1\supset B_1\supset A_2\supset\dots\supset A_n\supset B_n\supset\dots$, which is called a {\em play}  and will be denoted by $p=(A_i,B_i)_{i\ge 1}$. The player $\Omega$ is declared the winner of
\begin{itemize}
\item {\em the fragmenting game} $G(X)$ if the set $\bigcap_{i\ge 1}A_i$ contains at most one point;
\item {\em the strict fragmenting game} $G'(X)$ if the set $\bigcap_{i\ge 1}A_i$ is either empty or contains exactly one point $x$ such that each neighborhood $U\subset X$ of $x$ contains all but finitely many sets $A_i$, $i\ge 1$;
\item {\em the determination game} $DG(X)$ if the set $K=\bigcap_{i\ge 1}\bar{A_i}$ is compact and either $K$ is empty or each neighborhood $U\subset X$ of $K$ contains all but finitely many sets $A_i$, $i\ge 1$.
\end{itemize}

\begin{definition}[Kenderov-Kortezov-Moors] A topological space $X$ is defined to be {\em game determined} if the player $\Omega$ has a winning strategy in the determination game $DG(X)$.
\end{definition}

It turns out that the $\Omega$-favorability of the (strict) fragmenting game characterizes (strict) fragmentable spaces, see  \cite{KM2}, \cite{KM3}, \cite{KKM}, \cite{CM}.

\begin{theorem}[Kenderov, Moors]\label{t:game} A regular topological space is (strictly) fragmentable if and only if the player $\Omega$ has a winning strategy in the (strict) fragmenting game $G(X)$ (resp. $G'(X)$).
\end{theorem}

On the other hand, the absence of a winning strategy for the player $\Sigma$ in the (strict) fragmenting game characterizes weak Stegall (weak Piotrowski) spaces, see \cite{KKM}, \cite{CM}. A topological space $X$ is defined to be {\em weak Stegall} (resp. {\em weak Piotrowski}) if $X$ is $\C$-Stegall (resp. $\C$-Piotrowski) for the class $\C$ of complete metric spaces.

\begin{theorem}[Kenderov, Kortezov, Moors] A regular topological space is weak Stegall (resp. weak Piotrowski) if and only if the player $\Sigma$ has no winning strategy in the (strict) fragmenting game $G(X)$ (resp. $G'(X)$).
\end{theorem}

We shall use Theorem~\ref{t:game} to show that the class of strictly fragmentable spaces contains all $\sigma$-spaces. A topological spaces $X$ is called a {\em $\sigma$-space} if it is regular and possesses a $\sigma$-discrete network, see \cite[\S4]{Grue}.
The following theorem generalizes Proposition 2.1 in \cite{KM12} (saying that cosmic spaces are  strictly fragmentable).

\begin{theorem} Each $\sigma$-space $X$ is strictly fragmentable and hence Piotrowski.
\end{theorem}

\begin{proof} By definition, the $\sigma$-space $X$ has a $\sigma$-discrete network
 $\mathcal N=\bigcup_{k\in\w}\mathcal N_k$ (here $\mathcal N_k$ are discrete families in $X$). Replacing each set $N\in\mathcal N$ by its closure $\bar N$ in the regular space $X$, we can assume that the network $\mathcal N$ consists of closed subsets of $X$.

 By Theorem~\ref{t:game}, to show that $X$ is strictly fragmentable, it suffices to describe a wining strategy for the player $\Omega$ in the strict fragmenting game $G'(X)$.
Given the $n$th move $B_n$ of $\Sigma$, the player $\Omega$ consider the relatively open subset $U_n=B_n\setminus \bigcup\mathcal N_n$ of $B_n$. If $U_n\ne\emptyset$, then $\Omega$ answers by the set $A_n=U_n$. If the set $U_n$ is empty, then $\Omega$ answers with the set $A_n=B_n\cap N$, where $N\in\mathcal N_n$ is any set such that $B_n\cap N\ne\emptyset$. The set $A_n$ is relatively open in $B_n$, being the complement of the relatively closed set $B_n\cap\bigcup (\mathcal N_n\setminus\{N\})$.

Let us show that this strategy of the player $\Omega$ is winning. Let $(A_n,B_n)_{n\ge 1}$ be a play of the game $G'(X)$ where $\Omega$ plays according to the strategy described above.
Assume that the intersection $\bigcap_{n\in\IN}A_n$ contains some point $x\in X$ and let $O_x\subset X$ be any neighborhood of $x$. Find a set $N\in\mathcal N$ with $x\in N\subset O_x$ and a number $k\in\w$ with $N\in\mathcal N_k$. It follows from $x\notin U_k=B_k\setminus\bigcup\mathcal N_k$ that $x\in A_k=B_k\cap N'$ for some $N'\in\mathcal N_k$. Taking into account that $x\in N\cap N'$ and the family $\mathcal N_k\ni N,N'$ is discrete, we conclude that $N=N'$ and hence $A_k\subset N\subset O_x$. This means that the sequence $(A_k)_{k\in\IN}$ converges to the point $x$ and by the Hausdorff property of $X$, $\bigcap_{k\in\IN}A_k=\{x\}$.
Therefore, the strategy of $\Omega$ is winning in the strong fragmenting game $G'(X)$ and the space $X$ is strictly fragmented.
\end{proof}

Theorem~\ref{t:game} implies that each regular strictly fragmentable space is game determined. Moreover, according to \cite[Proposition 3]{KKM} we have the following characterization.

\begin{theorem}[Kenderov, Kortezov, Moors]\label{t:sf<=>f+gd} A regular topological space $X$ is strictly fragmentable if and only if it is fragmentable and game determined.
\end{theorem}

Applying game determined spaces to $\C$-Piotrowski spaces, we get the following characterization.

\begin{theorem}\label{t:P<=>S} Let $\C$ be a class of Baire spaces, closed under taking dense $G_\delta$-sets.
A game determined Tychonoff space $X$ is (strong) $\C$-Piotrowski if and only if $X$ is (strong) $\C$-Stegall.
\end{theorem}

\begin{proof} The ``only if'' part follows from Theorem~\ref{t:P->S}. To prove the ``if'' part, assume that the game determined Tychonoff space $X$ is (strong) $\C$-Stegall. To show that $X$ is (strong) $\C$-Piotrowski, take any quasicontinuous map $\varphi:Z\to X$ defined on a non-empty space $Z\in\C$.

Let $bX$ be any compactification of $X$ and $\bar \varphi$ be the closure of the graph $\{(z,\varphi(z)):z\in Z\}\subset Z\times bX$ of $\varphi$ in $Z\times bX$. The set $\bar \varphi$ can be thought as a multi-valued map $\bar \varphi:Z\multimap bX$ assigning to each point $z\in Z$ the set $\bar \varphi(z)=\{x\in bX:(z,x)\in\bar\varphi\}$.  Using Lemma~\ref{l:musco}, it can be shown that the quasicontinuity of the map $\varphi$ implies that $\bar \varphi$ is a minimal usco map.  By Theorem 6 \cite{KKM}, the set $\{z\in Z:\bar \varphi(z)\subset X\}$ contains a dense $G_\delta$-subset $G\subset Z$. Then the map $\bar \varphi|G:G\multimap X$ is a well-defined usco map into $X$. By Lemma~\ref{l:musco-rest}, the usco map $\bar \varphi |G:G\multimap X$ is minimal. Since the space $X$ is (strong) $\C$-Stegall, the map $\bar \varphi$ is single-valued at all points of some non-empty (comeager) set $C\subset G\in\C$.  The upper semicontinuity of the usco map $\bar \varphi$ implies the continuity of the map $\varphi$ at each point of the set $C$. This means that the space $X$ is (strong) $\C$-Piotrowski.
\end{proof}

By Theorems~\ref{t:sf<=>f+gd} and \ref{t:P<=>S}, for any game determined Tychonoff space  the following implications and equivalences hold:
$$
\xymatrix{
\mbox{strictly fragmentable}\ar@{=>}[r]\ar@{<=>}[d]&\mbox{Piotrowski}\ar@{=>}[r]\ar@{<=>}[d]&\mbox{weak Piotrowski}\ar@{<=>}[d]\\
 \mbox{fragmentable}\ar@{=>}[r]&\mbox{Stegall}\ar@{=>}[r]&\mbox{weak Stegall}.
}
 $$

The (consistent) counterexamples constructed in \cite{K99}, \cite{GKMS}, and \cite{MS} show that the horizontal implications in the above diagram cannot be reversed even for compact Hausdorff spaces.

It is clear that each compact Hausdorff space is game determined. On the other hand, the Sorgenfrey line is not game determined, see \cite[Corollary 6]{KKM}. This fact can be alternatively derived from the following two (known) properties of the Sorgenfrey line.

We recall that the {\em Sorgenfrey line} $\IS$ is the real line endowed with the topology generated by the base consisting of half-intervals $[a,b)$, $a<b$.

\begin{example}\label{e:sorg} The Sorgenfrey line fragmentable (and hence Stegall) but not (weak) Piotrowski.
\end{example}

\begin{proof} Observe that the real line $\IR$ is a complete metric space and the identity map $\IR\to\IS$ is quasicontinuous but has no continuity points, witnessing that $\IS$ is not (weak) Piotrowski. The Sorgenfrey line is fragmented by the standard Euclidean metric. So, it is fragmentable and hence Stegall.
\end{proof}

As was observed in \cite{KKM}, the class of game determined spaces is very wide. Besides strongly fragmentable regular spaces, it contains all Tychonoff spaces with countable separation, discussed in the next section.

\section{Spaces with countable separation}

\begin{definition}[Kenderov-Kortezov-Moors]\label{d:CS} A Tychonoff space $X$ is defined to have {\em countable separation} if some compactification $bX$ of $X$ contains a countable family $\U$ of open subsets such that for any points $x\in X$ and $y\in K\setminus X$ some set $U\in\U$ contains exactly one point of the set
$\{x,y\}$. In this case we shall say that the family $\U$ separates the points of the sets $X$ and $bX\setminus X$.
\end{definition}

The following important result was proved in \cite[Proposition 2]{KKM}

\begin{theorem}[Kenderov-Kortezov-Moors]\label{t:GD} Each Tychonoff space with countable separation is game determined.
\end{theorem}

This theorem motivates a deeper study of spaces with countable separation. We start with the following characterization of such spaces.

\begin{proposition} For a Tychonoff space  $X$ the following conditions are equivalent:
\begin{enumerate}
\item $X$ has countable separation;
\item for any Tychonoff space $Y$ containing $X$ there exists a countable family $\U$ of open subsets of $Y$ such that for any points $x\in X$ and $y\in Y\setminus X$ some set $U\in\U$ contains exactly one point of the doubleton $\{x,y\}$;
\item for some Tychonoff space $Y$ with countable separation that contains $X$, there exists a countable family $\U$ of open subsets of $Y$ such that for any points $x\in X$ and $y\in Y\setminus X$ some set $U\in\U$ contains exactly one point of the doubleton  $\{x,y\}$.
\end{enumerate}
\end{proposition}

\begin{proof} The implications $(1)\Ra(3)\Leftarrow (2)$ are trivial.
\smallskip

$(3)\Ra(1)$ Assume that for  some Tychonoff space $Y$ with countable separation that contains $X$, there exists a countable family $\V$ of open subsets of $Y$ separating the points of the sets $X$ and $Y\setminus X$. Since the space $Y$ has countable separation, some compactification $bY$ of $Y$ contains a countable family $\mathcal W$ separating the points of the sets $Y$ and $b Y\setminus Y$.

It follows that the closure $\bar X$ of $X$ in $bY$ is a compactification of $X$ and $\U=\{\bar X \cap V:V\in\V\}\cup\{\bar X\cap W:W\in\mathcal W\}$ is a countable family of open sets in $\bar X$ separating points of the sets $X$ and $\bar X\setminus X$, and witnessing that the space $X$ has countable separation.
\smallskip

$(1)\Ra(2)$ Assume that the space $X$ has countable separation. Then $X$ has a compactification $bX$ containing a countable family $\V$ of open sets, separating points of the sets $X$ and $bX\setminus X$. Let $f:\beta X\to bX$ be the canonical map of the Stone-\v Cech compactification $\beta X$ of $X$ onto $bX$. It follows that $f^{-1}(X)=X$ and hence $\V'=\{f^{-1}(V):V\in\V\}$ is a countable family of open sets separating the points of the sets $X$ and $\beta X$. Therefore, we lose no generality assuming that $bX=\beta X$. Let $Y$ be any Tychonoff space containing $X$ and let $\beta Y$ be the Stone-\v Cech compactification of $Y$. Then the closure $\bar X$ of $X$ in $\beta Y$ is a compactification of $X$. The inclusion map $g:X\to \bar X$ admits a continuous extension $\bar g:\beta X\to \bar X\subset \beta Y$. It can be shown that the countable family of open sets $\U=\{Y\setminus\bar X\}\cup\{Y\setminus \bar g(\beta X\setminus V):V\in\mathcal V\}$ of $Y$ separates points of the sets $X$ and $Y\setminus X$.
\end{proof}

The class of Tychonoff spaces with countable separation is very wide: for every compact Hausdorff space $K$ the family of subspaces with a countable separation in $K$ is a $\{-1,1\}^\w$-algebra.

A family $\A$ of subsets of a set $X$ is defined to be a {\em $\{-1,1\}^\w$-algebra} if for any sequence of sets $(A_n)_{n\in\w}\in\A^\w$ and any set  $\Omega\subset \{-1,1\}^\w$ the set
$$\Omega(A_n)_{n\in\w}=\bigcup_{f\in\Omega}\bigcap_{n\in \w}f(n)\cdot A_{n}$$belongs to $\A$. Here $1\cdot A_n=A_n$ and $(-1)\cdot A_n=X\setminus A_n$. The set $\Omega(A_n)_{n\in\w}$ will be called {\em the result of the $\Omega$-operation over the sequence} $(A_n)_{n\in\w}$.

The $\Omega$-operations generalize many known operations over sequences of sets.
In particular, for the sets $\Omega_i=\{f\in\{-1,1\}^\w:i\in f(\w)\}$, $i\in\{-1,1\}$, we get $\Omega_1((A_n)_{n\in\w})=\bigcup_{n\in\w}A_n$ and $\Omega_{-1}((A_n)_{n\in\w})=\bigcup_{n\in\w}(X\setminus A_n)$. This means that each $\{-1,1\}^\w$-algebra of subsets of a set $X$ is a $\sigma$-algebra of subsets of $X$.

The family of $\Omega$-operations includes also the classical Suslin $A$-operation
$$\bigcup_{\alpha\in\w^\w}\bigcap A_{\alpha|n}$$over a sequence $(A_s)_{s\in \w^{<\w}}$ of sets. Here $\w^{<\w}=\bigcup_{n\in\w}\w^n$ is the family of finite sequences of finite ordinals (which includes the restriction $\alpha|n$ of any function $\alpha\in\w^\w$ to a finite ordinal $n\in\w$).

To represent the Suslin $A$-operation as an $\Omega$-operation, take any bijection $\xi:\w^{<\w}\to \w$ and consider the set  $$\Omega=\bigcup_{\alpha\in\w^\w}\{\beta\in\{-1,1\}^\w:
\forall n\in\w\;\;\beta(\xi(\alpha|n))=1\}.$$ Observe that for any sequence $(A_s)_{s\in\w^{<\w}}$ of subsets of $X$ we get
$$\bigcup_{\alpha\in\w^\w}\bigcap_{n\in\w}A_{\alpha|n}=\Omega(B_n)_{n\in\w}$$where $B_n=A_{\xi^{-1}(n)}$ for all $n\in\w$.

\begin{proposition} For every compact Hausdorff space $K$ the family of subspaces with countable separation in $K$ is a $\{-1,1\}^\w$-algebra, containing all open subsets of $K$. Consequently, this family is a $\sigma$-algebra of subsets of $K$, which contains all Borel subsets of $K$ and is closed under Suslin $A$-operations.
\end{proposition}

\begin{proof} It is clear that each open subspace of $K$ has countable separation. Let $(A_n)_{n\in\w}$ be a sequence of subspaces with countable separation in $K$. Choose a countable family $\U$ of open subsets of $K$ such that for any $n\in\w$ and any points $x\in A_n$, $y\in K\setminus A_n$ there exists a set $U\in\U$ containing exactly one point of the doubleton $\{x,y\}$. Given any subset $\Omega\subset\{-1,1\}^\w$ consider the result $A=\Omega(A_n)_{n\in\w}$ of the $\Omega$-operation over the sequence $(A_n)_{n\in\w}$. Given any points $x\in A$ and $y\in K\setminus A$, find a function $f\in\Omega$ such that $x\in\bigcap_{n\in\w}f(n)\cdot A_n\subset A$. Since $y\notin\bigcap_{n\in\w}f(n)\cdot A_n$, there exists $n\in\w$ such that $y\notin f(n)\cdot A_n$ and $x\in f(n)\cdot A_n$. By the choice of the family $\U$, there exists a set $U\in\U$ containing exactly one point of the doubleton $\{x,y\}$. This means that the countable family $\U$ separates the points of the sets $A$ and $K\setminus A$, and hence the space $A$ has countable separation.
\end{proof}

Let us also prove the following stability property of the class of spaces with countable separation.

\begin{proposition}\label{p:CS-map} Let $f:X\to Y$ be a continuous map between Tychonoff spaces.
\begin{enumerate}
\item If the space $X$ has countable separation, then for every subspace $Z\subset Y$ with countable separation the preimage $f^{-1}(Z)$ has countable separation.
\item If $Y$ has countable separation and $f$ is perfect, then the space $X$ has countable separation, too.
\end{enumerate}
\end{proposition}

\begin{proof} Consider the (unique) extension of $f$ to a continuous map $\bar f:\beta X\to\beta Y$ to the Stone-\v Cech compactifications of the spaces $X,Y$.
\smallskip

1. If the space $X$ has countable separation, then some countable family $\U_X$ of open subsets of $\beta X$ separates the points of the sets $X$ and $\beta X\setminus X$. Assuming that a subspace $Z\subset Y$ has countable separation, we can find a countable family $\U_Z$ of open subsets of $\beta Y$ separating the  points of the sets $Z$ and $Y\setminus Z$. Then the countable family $\U
=\U_X\cup\{\bar f^{-1}(U):U\in\U_Z\}$ separates points of $f^{-1}(Z)$ and $\beta X\setminus f^{-1}(Z)$, witnessing that the space $f^{-1}(Z)$ has countable separation.
\smallskip

2. Assuming that $Y$ has countable separation and $f$ is perfect, we conclude that $X=\bar f^{-1}(Y)$. By the preceding statement, the space $X=\bar f^{-1}(Y)$ has countable separation.
\end{proof}

Example~\ref{e:sorg} and Theorems~\ref{t:P<=>S}, \ref{t:GD} imply:

\begin{example} The Sorgenfrey line fails to have countable separation.
\end{example}

Many examples of Tychonoff spaces without countable separation can be constructed using the following proposition.

\begin{proposition}\label{p:no-cs} If a Tychonoff space $(X,\tau)$ has cardinality $|X|\ge|\tau^\w|>\mathfrak c$, then $X$ contains a subspace $Y\subset X$ without countable separation.
\end{proposition}

 \begin{proof} Let $\kappa=|\tau^\w|$ and let $\{f_\alpha\}_{\alpha\in\kappa}$ be an enumeration of the set $\tau^\w$ of all countable sequences of open subsets of $X$.

For every $U\in\tau$ let $\chi_U:X\to\{0,1\}$ be the characteristic function of the set $U$ in $X$. For every $\alpha\in\kappa$ consider the function $\hat f_\alpha:X\to\{0,1\}^\w$ assigning to each point $x\in X$ the sequence $\big(\chi_{f_\alpha(n)}(x)\big)_{n\in\w}\in\{0,1\}^\w$. By transfinite induction we shall construct two transfinite sequences of points $(x_\alpha)_{\alpha\in\kappa}$ and $(y_\alpha)_{\alpha\in\kappa}$ in $X$ such that $\hat f_\alpha(x_\alpha)=\hat f_\alpha(y_\alpha)$, $x_\alpha\ne y_\alpha$ and $x_\alpha,y_\alpha\subset X\setminus\{x_\beta,y_\beta\}_{\beta<\alpha}$ for every $\alpha<\kappa$.

  Assume that for some $\alpha<\kappa$ the points $x_\beta,y_\beta$, $\beta<\alpha$, have been constructed. Consider the map $\hat f_\alpha:X\to\{0,1\}^\w$. Since $|X|\ge |\tau^\w|=\kappa>\mathfrak c=|\{0,1\}^\w|$, there is a point $z\in\{0,1\}^\w$ such that $|\hat f_\alpha^{-1}(z)|\ge \kappa$. Then we can choose two distinct points $x_\alpha,y_\alpha\in\hat f_\alpha^{-1}(z_\alpha)\setminus\{x_\beta,y_\beta\}_{\beta<\alpha}$. This completes the inductive step.

 We claim that the subspace $Y=\{y_\alpha\}_{\alpha\in \kappa}$ of $X$ has no countable separation. To derive a contradiction, consider any compactification $bX$ of the Tychonoff space $X$. Assuming that $X$ has countable separation, we can find a sequence $(U_n)_{n\in\w}$ of open sets in $bX$ separating the points of the sets $Y$ and $bX\setminus Y$. Find $\alpha\in\kappa$ such that $f_\alpha(n)=U_n\cap X$ for all $n\in\w$. The choice of the points $x_\alpha\in bX\setminus Y$ and $y_\alpha\in Y$ guarantees that they cannot be separated by the sequence $(U_n)_{n\in\w}$. This contradiction completes the proof.
 \end{proof}

Proposition~\ref{p:no-cs} will help us to construct a metrizable spaces without countable separation.

\begin{corollary} The class of game determined spaces contains a metrizable space that fails to have countable separation.
\end{corollary}

\begin{proof} We recall that $\beth_\w=\sup_{n\in\w}\beth_n$ where $\beth_0=\w$ and $\beth_{n+1}=2^{\beth_n}$ for $n\ge 0$. Endow the cardinal $\beth_\w$ with the discrete topology and consider the completely metrizable space $X=(\beth_\w)^\w$. This spaces has weight $w(X)=\beth_\w>\beth_1=\mathfrak c$ and hence the topology $\tau$ of $X$ has cardinality
$$|\tau|\le 2^{\beth_\w}=2^{\sum_{n\in\w}\beth_n}=\prod_{n\in\w}2^{\beth_n}=
\prod_{n\in\w}\beth_{n+1}\le|\beth_\w^\w|$$
 and hence $\mathfrak c<|\tau^\w|\le|\beth_\w^\w|=|X|$. By Proposition~\ref{p:no-cs}, the space $X$ contains a subspace $Y$ without countable separation. The metrizable space $Y$ is strictly fragmentable and hence is game determined.
 \end{proof}

\begin{proposition} Each metrizable space $X$ of density $\le\mathfrak c$ has countable separation.
\end{proposition}

\begin{proof} Let $\bar X$ be the completion of $X$ with respect to any fixed metric generating the topology of $X$. The completely metrizable space $\bar X$, being \v Cech complete, has countable separation. Theorem 4.4.9 \cite{En} implies that the metrizable space $\bar X$ of density $\le\mathfrak c$ admits a continuous injective map $f:\bar X\to Y$ to a metrizable separable space $Y$. Consider the image $Z=f(X)$ and observe that the metrizable separable space $Z$ has countable separation. By Proposition~\ref{p:CS-map}(1), the preimage $X=f^{-1}(Z)$ has countable separation.
\end{proof}

 \begin{problem} What is the smallest density of a metrizable space without countable separation? Is it equal to $\mathfrak c^+$? \footnote{This problem is discussed (but not solved) at (http://mathoverflow.net/questions/243064/what-is-the-smallest-density-of-a-metrizable-space-without-countable-separation).} \end{problem}

\section{Dense metrizable subsets in Piotrowski spaces}\label{s:DM}

In this section we construct dense (completely) metrizable subsets in Baire (Choquet) spaces which are strictly fragmentable, Piotrowski, or Stegall.
The following theorem (generalizing Proposition 6 of \cite{KKM}) should be known but we could not find a precise reference in the literature.

\begin{theorem} Each Baire strictly fragmentable space $X$ contains a metrizable dense $G_\delta$-set.
\end{theorem}

\begin{proof} Let $d$ be a metric that strictly fragments the topology of $X$. For every $n\in\w$ let $\U_n$ be a maximal disjoint family of open subsets of $d$-diameter $<2^{-n}$. The choice of the metric $d$ guarantees that $\bigcup\U_n$ is dense in $X$. Replacing each family $\U_n$, $n\ge 1$ by the family $\{U\cap V:U\in\U_n,\;V\in\U_{n-1}\}$ we can assume that each set $U\in\U_n$ is contained in some set $V\in\U_{n-1}$. Since the space $X$ is Baire, the $G_\delta$-set $G=\bigcap_{n\in\w}\bigcup\U_n$ is dense in $X$. The choice of the families $\U_n$, $n\in\w$, guarantees that the topology on $G$ induced from $X$ is generated by the fragmenting metric $d$. Therefore, $G$ is a metrizable dense $G_\delta$-set in $X$.
\end{proof}

\begin{theorem}\label{t:Baire} Assume that a Tychonoff space $X$ is $\C$-Piotrowski for the class $\C$ of Baire metrizable spaces of density $\le w(X)$. The space $X$ contains a dense metrizable Baire subspace if and only if $X$ is Baire.
\end{theorem}

\begin{proof} The ``only if'' part is trivial. To prove the ``if'' part, assume that the space $X$ is Baire. By \cite[2.3.23]{En}, the Tychonoff space $X$ has a compactification $\bar X$ of weight $w(\bar X)=w(X)$. Then the Banach space $C(\bar X)$ of all real-valued continuous functions on $\bar X$ has density $w(\bar X)=w(X)$.
By our assumption, every Baire subspace of $C(\bar X)$ belongs to the class $\C$.   Consider the multi-valued map $\Phi:C(\bar X)\multimap \bar X$ assigning to each function $f\in C(\bar X)$ the non-empty compact set $\Phi(f)=\{x\in \bar X:f(x)=\max f(\bar X)\}$.
By Lemma~\ref{l:CK-musco}, $\Phi$ is a minimal usco map such that for every open set $V\subset C(\bar X)$ the set $\Phi(V)$ is open in $\bar X$. It is easy to see that  $Z=\{f\in C(X):\Phi(f)\cap X\ne\emptyset\}$ is a dense subspace in $C(\bar X)$. Let $L\subset C(\bar X)$ be the largest open subset of $C(\bar X)$ such that the intersection $L\cap Z$ is meager in $C(\bar X)$. Find a sequence $(M_n)_{n\in\w}$ of nowhere dense subsets of $C(\bar X)$ such that $L\cap Z\subset\bigcup_{n\in\w}M_n$.
%In the opposite case we can find a non-empty open set $U\subset C(bX)$ such that $U\cap Z$ is contained in a meager $F_\sigma$-subset $M$ of $C(bX)$.

%Write the set $M$ as the countable union $M=\bigcup_{n\in\w}M_n$ of closed nowhere dense subsets in $C(bX)$. By Lemma~2.2.1(iv) in \cite{Fab}, for every $n\in\w$ the set $U_n=\Phi(C(bX)\setminus M_n)$ is open in $bX$. We claim that the set $U_n$ is dense in $bX$. In the opposite case, we could find a non-empty open set $V\subset bX$, disjoint with $U_n$. It is easy to see that the set $\tilde V=\{f\in C(K):\Phi(f)\subset V\}$ is non-empty and open in $C(K)$ (by the upper semicontinuity of the usco map $\Phi$). Since the set $M$ is measger in the complete metric space $C(K)$, there exists a function $f\in \tilde V\setminus M$. For this function $f$, we get $\Phi(f)\subset V\cap U_n=\emptyset$, which is a desired contradiction, showing that the open set $U_n$ is dense in $bX$.

We are going to construct a sequence $(\mathcal H_n)_{n\in\w}$ of non-empty disjoint families of open sets in $C(\bar X)$ satisfying the following conditions:
\begin{enumerate}
\item $\bigcup_{H\in\mathcal H_n}\Phi(H)$ is dense in $\bar X$;
\item $\bigcup\mathcal H_n\subset C(\bar X)\setminus M_n$;
\item each set $H\in\mathcal H_n$ has diameter $\le 2^{-n}$ in the Banach space $C(\bar X)$;
\item for any distinct sets $H,H'\in\mathcal H_n$ the sets $\Phi(H)$ and $\Phi(H')$ are disjoint;
\item for every $H\in\mathcal H_{n}$ there exists a unique set $H'\in\mathcal H_{n-1}$ such that $\bar H\subset H'$;
\item if $L$ is not empty, then $L$ contains the closure of some set $H\in\mathcal H_0$.
\end{enumerate}
Here we assume that $\mathcal H_{-1}=\{C(\bar X)\}$. The construction of the sequence $(\mathcal H_n)_{n\in\w}$ is inductive. Assume that for some $n\ge 0$ a family $\mathcal H_{n-1}$ satisfying the condition (1) have been constructed. Using the Zorn Lemma, we can choose a maximal family $\mathcal H_n$ of non-empty open sets in $C(\bar X)$ satisfying the conditions (2)--(6). We claim that the family $\mathcal H_n$ satisfies the condition (1), too. Assuming that the union $\bigcup_{H\in\mathcal H_n}\Phi(H)$ is not dense in $\bar X$, we conclude that this union is disjoint with some non-empty open set $W\subset \bar X$. By the inductive assumption, the open set $\bigcup_{H\in\mathcal H_{n-1}}\Phi(H)$ is dense in $\bar X$. So, we can replace $W$ by a smaller open set and assume that $W\subset \Phi(H)$ for some set $H\in\mathcal H_{n-1}$. By Lemma~\ref{l:musco}, the minimality of the usco map $\Phi$ guarantees that $\Phi(J)\subset W$ for some non-empty open set $J\subset H$. Replacing $J$ by a smaller open set, we can additionally assume that $\diam(J)\le 2^{-n}$ and $\bar J\subset H\setminus M_n$. Then the family $\mathcal H_n\cup\{J\}$ satisfies the conditions (2)--(6), which contradicts the maximality of $\mathcal H_n$. This contradiction completes the proof of the condition (1) of the inductive construction.
\smallskip

After completing the inductive construction, consider the $G_\delta$-set $G=\bigcap_{n\in\w}\bigcup\mathcal H_n$ in $C(\bar X)$ and the dense $G_\delta$-set $G'=\bigcap_{n\in\w}\bigcup_{H\in\mathcal H_n}\Phi(H)$ in $\bar X$. We claim that for every $x\in G'$ there exists a  function $\psi_x\in G$ such that $x\in\Phi(\psi_x)$. Given any $x\in G'$, for every $n\in\w$ choose a function $h_n\in H_n$ with $x\in \Phi(h_n)$. The completeness of the Banach space $C(\bar X)$ and the conditions (3,5) imply that the sequence $(h_n)_{n\in\w}$ is Cauchy and hence it converges to a unique function $\psi_x\in \bigcap_{n\in\w}\bar H_n=\bigcap_{n\in\w}H_n\subset G$.
The upper semicontinuity of the usco map $\Phi$ guarantees that $x\in\Phi(\psi_x)$.
The conditions (3,4) imply that the map $\psi:G'\to G$, $\psi:x\mapsto \psi_x$, is continuous.

Now we can show that the dense subspace $Z$ of $C(\bar X)$ is Baire. In the opposite case, the set $L$ is not empty and by condition (6) $L$ contains the closure $\bar H_0$ of some set $H_0\in\mathcal H_0$. It follows that $\Phi(H_0)$ is a non-empty open set in $\bar X$. Since $X$ is a dense Baire subspace of $\bar X$, the intersection $X\cap G'\cap \Phi(H_0)$ contains some point $x$. Consider the function $\psi_x\in \bar H_0\cap G\subset L$ and observe that the conditions (2),(5) imply that $\psi_x\in L\setminus\bigcup_{n\in\w}M_n\subset L\setminus Z$, which is not possible as $x\in \Phi(\psi_x)$ and hence $\psi_x\in Z=\{z\in C(\bar X):\Phi(z)\cap X\ne\emptyset\}$. This contradiction shows that $L=\emptyset$ and the subspace $Z$ of $C(\bar X)$ is Baire.

For every $z\in Z$ choose any point $\varphi(z)\in\Phi(z)\cap X$. By Lemma~\ref{l:musco-rest}, the restriction $\Phi|Z:Z\multimap \bar X$ is a minimal usco and by Lemma~\ref{l:select} its selection $\varphi:Z\to X$ is a quasicontinuous function. Since $Z\in\C$ and the space $X$ is (strong) $\C$-Piotrowski, the map $\varphi$ is continuous at every point of some dense $G_\delta$-set $D\subset Z$. Repeating the proof of Theorem~\ref{t:P->S} (or using Lemma~\ref{l:musco}), we can show that $\Phi(z)=\{\varphi(z)\}$ for every $z\in D$.

 Now we show that the image $\Phi(D)$ is a dense Baire subspace of $X$. In the opposite case we could find a non-empty open set $U\subset \bar X$ such that the set $U\cap \Phi(D)$ is meager and hence is contained in the countable union  $\bigcup_{n\in\w}F_n$ of closed nowhere dense subsets of $\bar X$. Observe that the set $\tilde U=\{z\in C(\bar X):\Phi(z)\subset U\}$ is non-empty and open in $C(\bar X)$ (by the upper semicontinuity of $\Phi$).

The upper semicontinuity of the map $\Phi$ implies that for every $n\in\w$ the set $E_n=\{z\in \tilde U:\Phi(z)\cap F_n\ne\emptyset\}$ is closed in $\tilde U$.
We claim that this set is nowhere dense in $C(\bar X)$. Assuming the opposite, we can consider the interior $E^\circ_n$ of $E_n$ in $C(\bar X)$ and using Lemma~\ref{l:CK-musco}, conclude that $\Phi(E^\circ_n)$ is a non-empty open set in $\bar X$.
 Since the set $F_n$ is nowhere dense in $\bar X$, the complement $\Phi(E^\circ_n)\setminus F_n$ is not empty. Applying Lemma~\ref{l:musco}, find a non-empty open set $V\subset E_n^\circ$ such that $\Phi(V)\subset \Phi(E^\circ_n)\setminus F_n$. But this contradicts the inclusion $V\subset E_n=\{z\in \tilde U:\Phi(z)\cap F_n\ne\emptyset\}$. This contradiction shows that each the set $E_n$ is nowhere dense in $\tilde U$ and then the set $D\cap\tilde U\subset \bigcup_{n\in\w}E_n$ is meager in $C(\bar X)$, which is a contradiction, showing that the image $\Phi(D)$ is a dense Baire subspace of $X$.

 Since $G'$ is a dense $G_\delta$-set in $\bar X$, the intersection $B'=G'\cap \Phi(D)$ is a dense Baire subspace of $X$. Now consider the subspace $B=\psi(B')$ of $G$. Taking into account that the restriction $\Phi|D$ is single-valued, we conclude that $\varphi\circ\psi|B'$ is the identity map of $B'$, which means that $\psi:B'\to B$ is a homeomorphism and hence the dense Baire subspace $B'$ of $X$ is metrizable.
 \end{proof}

Next, we shall characterize Piotrowski spaces containing a dense completely metrizable subspaces and prove that those are exactly Choquet spaces (which are defined with the help of the classical Choquet game).

The {\em Choquet game} on a topological space $X$ is played by two players, $\alpha$ and $\beta$, who select consecutively non-empty open subsets of $X$. The player $\alpha$ starts the Choquet game selecting the set $U_0=X$. The player $\beta$ answers by choosing a non-empty open subset $V_0$ of $U_0$. At the $n$-th move the player $\alpha$ chooses a non-empty open set $U_n$ in the set $V_{n-1}$ given by $\beta$ at the $(n-1)$-th move, and then $\beta$ answers with  a non-empty open set $V_n\subset U_n$. At the end of the game the player $\alpha$ is declared the winner if $\bigcap_{n\in\w}U_n\ne\emptyset$. In the opposite case the player $\beta$ wins the game $BM(X)$.

By the classical result of Oxtoby \cite[8.11]{Ke95}, a topological space $X$ is Baire if and only if the player $\beta$ has no winning strategy in the Choquet game. A topological space $X$ is called {\em Choquet} if the player $\alpha$ has a winning strategy in the Choquet game on $X$. It is well-known (see e.g. \cite[8.33]{Ke95}) that a (metrizable) topological space $X$ is Choquet if (and only if) $X$ contains a dense completely metrizable subspace. By the following theorem, this characterization remains true also for Piotrowski spaces. A topological space is called {\em completely metrizable} if it is homeomorphic to a complete metric space.

\begin{theorem}\label{t:Choquet} Assume that a Tychonoff space $X$ is $\C$-Piotrowski for the class $\C$ of Baire metrizable spaces of density $\le w(X)$. The space $X$  contains a dense completely metrizable subspace if and only if $X$ is Choquet.
\end{theorem}

\begin{proof} The ``only if'' part is well-known (see, e.g. \cite[5.3]{Rev04}). To prove the ``if'' part, assume that the $\C$-Piotrowski space $X$ is Choquet.

By \cite[2.3.23]{En}, the Tychonoff space $X$ has a compactification $\bar X$ of weight $w(\bar X)=w(X)$. Then the Banach space $C(\bar X)$ has density $w(\bar X)=w(X)$ and hence every Baire subspace of $C(\bar X)$ belongs to the class $\C$.
Consider the multi-valued map $\Phi:C(\bar X)\to \bar X$ assigning to each function $f\in C(\bar X)$ the non-empty compact set $\Phi(f)=\{x\in \bar X:f(x)=\max f(\bar X)\}$.
By Lemma~\ref{l:CK-musco}, $\Phi$ is a minimal usco map such that for every open set $V\subset C(\bar X)$ the set $\Phi(V)$ is open in $\bar X$.
The proof of Theorem~\ref{t:Baire} implies that $Z=\{z\in C(\bar X):\Phi(z)\cap X\ne\emptyset\}$ is a dense Baire subspace of the Banach space $C(\bar X)$, containing a dense relative $G_\delta$-set $D\subset Z$ such that the restriction $\Phi|D$ is single-valued.

Now our aim is to construct a completely metrizable subspace $G\subset D$ such that  the image $\Phi(G)$ is a dense completely metrizable subspace of $X$.
Find a decreasing sequence $(D_n)_{n\in\w}$ of dense open subsets of $C(\bar X)$ such that $Z\cap\bigcap_{n\in\w}D_n=D$.

Let $\mathcal P$ be the family of all decreasing sequences $(U_0,V_0,\dots,U_{n-1},V_{n-1})$ of non-empty open sets of $X$ with $U_0=X$.
Since the topological space $(X,\tau)$ is Choquet, the player $\alpha$ has a winning strategy in the Choquet game on $X$. This strategy is a map $\$:\mathcal P\to\tau$ assigning to each partial play $(U_0,V_0,\dots,U_{n-1},V_{n-1})\in\mathcal P$ a non-empty open subset $U_n$ of $V_{n-1}$ such that for any decreasing sequence $U_0\supset V_0\supset U_1\supset V_1\supset\dots$ of non-empty open sets with $U_0=X$ and $U_n=\$(U_0,V_0,\dots,U_{n-1},V_{n-1})$ for all $n\in\w$ the intersection $\bigcap_{n\in\w}U_n$ is not empty.

Let $\bar\tau$ denote the topology of the compact Hausdorff space $\bar X$. Let $\tilde{\;}:\tau\to\bar\tau$ be the map assigning to each open set $U\in\tau$ the largest open set $\tilde U=\bar X\setminus\overline{X\setminus U}$ of $\bar X$ such that $\tilde U\cap X=U$.

Let $\tilde{\mathcal P}$ be the family of all decreasing sequences $(U_0,V_0,\dots,U_{n-1},V_{n-1})$ of non-empty open sets of $\bar X$ with $U_0=\bar X$.
The strategy $\$$ of the player $\alpha$ in the Banach-Mazur game induces a map $\tilde\$:\tilde{\mathcal P}\to\bar\tau$ defined by $\tilde\$(U_0,V_0,\dots,U_{n-1},V_{n-1})=\tilde U_n$ where $U_n=\$(U_0\cap X,V_0\cap X,\dots,U_{n-1}\cap X,V_{n-1}\cap X)$.

We are going to construct a sequence $(\mathcal H_n)_{n\in\IN}$ of  disjoint families of non-empty open sets in $C(\bar X)$ and a sequence of maps $(e_n:\mathcal H_n\to\bar\tau)_{n\in\IN}$ such that for every $n\in\IN$ the following conditions are satisfied:
\begin{enumerate}
\item $\bigcup_{H\in\mathcal H_n}\Phi(H)$ is dense in $\bar X$;
\item $\bigcup\mathcal H_n\subset D_n$;
\item each set $H\in\mathcal H_n$ has diameter $\le 2^{-n}$ in the Banach space $C(\bar X)$;
\item for any distinct sets $H,H'\in\mathcal H_n$ the sets $\Phi(H)$ and $\Phi(H')$ are disjoint;
\item for every $H\in\mathcal H_{n}$ there exists a unique set $H_{n-1}\in\mathcal H_{n-1}$ such that $\bar H\subset H_{n-1}$;
\item for every decreasing sequence of open sets $(H_k)_{k=0}^n\in \prod_{k=0}^n\mathcal H_k$ we get $\Phi(H_n)\subset U_n\subset e_n(H_n)\subset\Phi(H_{n-1})$  where $U_0=\bar X$ and $U_k:=\tilde\$(U_0,e_1(H_1),\dots,U_{k-1},e_{k}(H_{k}))$ for all $0\le k\le n$.
\end{enumerate}
Here we assume that $\mathcal H_{0}=\{C(\bar X)\}$. The construction of the sequences $(\mathcal H_n)_{n\in\IN}$ and $(e_n)_{n\in\IN}$ is inductive. Assume that for some $n\in\IN$ we have constructed families $\mathcal H_{i}$, $i<n$, and maps $e_i:\mathcal H_i\to\bar\tau$, $i<n$, satisfying the conditions (1)--(6). Using the Zorn Lemma, we can choose a maximal family $\mathcal H_n$ of non-empty open sets in $C(\bar X)$ possessing a map $e_n:\mathcal H_n\to\bar\tau$ satisfying the conditions (2)--(6). We claim that the family $\mathcal H_n$ satisfies the condition (1), too. Assuming that the union $\bigcup_{H\in\mathcal H_n}\Phi(H)$ is not dense in $\bar X$, we conclude that this union is disjoint with some non-empty open set $W\subset \bar X$. By the inductive assumption, the open set $\bigcup_{H\in\mathcal H_{n-1}}\Phi(H)$ is dense in $\bar X$. So, we can replace $W$ by a smaller open set and assume that $W\subset \Phi(H_{n-1})$ for some set $H_{n-1}\in\mathcal H_{n-1}$. By the condition (5), there exists a
decreasing sequence of open sets $(H_k)_{k=0}^{n-2}\in\prod_{k=0}^{n-2}\mathcal H_k$ such that $H_{n-1}\subset H_{n-2}$. The inductive assumption (6) guarantees that for all $k\le n-1$ we get $\Phi(H_k)\subset U_k\subset e_k(H_k)\subset\Phi(H_{k-1})$ where $U_0=\bar X$ and $U_i:=\tilde\$(U_0,e_1(H_1),\dots,U_{i-1},e_{i}(H_{i}))$ for $0<i<n$. In particular, $W\subset \Phi(H_{n-1})\subset U_{n-1}$ and hence the open set  $$U_n=\tilde\$(U_0,e_1(H_1),\dots,U_{n-2},e_{n-1}(H_{n-1}),U_{n-1},W)\subset W$$is well-defined.
 By Lemma~\ref{l:musco}, there exists a non-empty open set $H_n\subset H_{n-1}$ such that $\Phi(H_n)\subset U_{n}$. Replacing $H_n$ by a smaller open set, we can additionally assume
that $\diam(H_n)\le 2^{-n}$ and $\bar H_n\subset H_{n-1}\cap D_n$. Observe that $\Phi(H_n)\subset  U_n\subset W\subset \Phi(H_{n-1})$. Then the family $\mathcal H_n\cup\{H_n\}$ and the map $e_n\cup\{(H,W)\}$ satisfies the conditions (2)--(6), contradicting the maximality of $\mathcal H_n$. This contradiction completes the proof of the condition (1) of the inductive construction.

Now consider the $G_\delta$-set $G'=\bigcap_{n\in\w}\bigcup\mathcal H_n\subset \bigcap_{n\in\w}D_n$ in $C(\bar X)$ and the dense $G_\delta$-set $G=\bigcap_{n\in\w}\bigcup_{H\in\mathcal H_n}\Phi(H)$ in the compact Hausdorff space $\bar X$. We claim that $G'\subset Z$. Given any point $z\in G'$, use the conditions (3,5) to find a unique  decreasing sequence $(H_n)_{n\in\w}\in\prod_{n\in\w}\mathcal H_n$ such that $\{z\}=\bigcap_{n\in\w}H_n=\bigcap_{n\in\w}\bar H_n$.
 Next, consider the sequence of open sets $(U_n)_{n\in\w}$ in $\bar X$ defined by $U_0=X$ and $U_i=\tilde\$(U_0,e_1(H_1),\dots,U_{i-1},e_i(H_i))$ for $i\in\IN$. The choice of the (winning) strategy $\$$ guarantees that the intersection $X\cap\bigcap_{i\in\w}U_i$ is not empty. So, we can choose a point $\varphi(z)\in X\cap \bigcap_{i\in\w}U_i$ and observe that $\varphi(z)\in \bigcap_{i\in\w}U_i\subset\bigcap_{i=1}^\infty\Phi(H_{i-1})$. For every $i\in\w$, choose a point $z_i\in H_i$ with $\varphi(z)\in\Phi(z_i)$. The conditions (3,5) imply   that the sequence $(z_i)_{i\in\w}$ converges to the point $z$. Now the upper semicontinuity of the map $\Phi$ guarantees that $\varphi(z)\in\Phi(z)$ and hence $z\in Z$. This means that $G'\subset Z\cap\bigcap_{n\in\w}D_n=D$ and hence $\Phi(z)=\{\varphi(z)\}$ for every $z\in G'$. Then $G=\Phi(G')=\varphi(G')$. The upper semicontinuity of the map $\Phi$ implies the continuity of the map $\varphi:G'\to G$. The conditions (3),(4),(5) imply that the map $\varphi$ is bijective and the inverse map $\varphi^{-1}:G\to G'$ is continuous. Now we see that the space $X$ contains the dense $G_\delta$-set $G$, homeomorphic to the completely metrizable space $G'$, which completes the proof.
 \end{proof}

\begin{corollary} A Piotrowski Tychonoff space $X$ is
\begin{enumerate}
\item Baire if and only if $X$ contains a dense metrizable Baire subspace.
\item Choquet if and only if $X$ contains a dense completely metrizable subspace.
\end{enumerate}
\end{corollary}

\begin{problem} Does each weak Piotrowski Choquet Tychonoff space contain a dense metrizable subspace?
\end{problem}

A partial answer to this problem is given by the following proposition, proved in  \cite[Proposition 6]{KKM}.

\begin{proposition}[Kenderov-Kortezov-Moors] Every weak Piotrowski Choquet regular space $X$ contains a dense first-countable subspace.
\end{proposition}

Combining Theorems~\ref{t:Baire} and \ref{t:Choquet} with Theorem~\ref{t:P<=>S}, we obtain the following generalizations of the classical results of Stegall \cite{St87}, \cite{St91} and \v Coban, Kenderov \cite{CK}.

\begin{theorem} Let $X$ be a game determined Tychonoff space.
\begin{enumerate}
\item If $X$ is Baire and fragmentable, then $X$ contains a dense metrizable $G_\delta$-subset.
\item If $X$ is Baire and $\C$-Stegall for the class $\C$ of Baire metrizable spaces of density $\le w(X)$, then $X$ contains a dense metrizable Baire subspace.
\item If $X$ is Choquet and $\C$-Stegall for the class $\C$ of Baire metrizable  spaces of density $\le w(X)$, then $X$ contains a dense completely metrizable subspace.
\end{enumerate}
\end{theorem}

\begin{corollary} A game determined Stegall Tychonoff space is
\begin{enumerate}
\item Baire if and only if $X$ contains a dense metrizable Baire subspace.
\item Choquet if and only if $X$ contains a dense completely metrizable subspace.
\end{enumerate}
\end{corollary}

\section{Stability properties of the classes of (strong) $\C$-Stegall spaces}

In this section we discuss stability properties of the classes of (strong) $\C$-Stegall spaces.

\begin{proposition} Let $\C$ be a class of Baire spaces closed under taking open subspaces.
A topological space $X$ is (strong) $\C$-Stegall if one of the following conditions is satisfied:
\begin{enumerate}
\item $X$ is a subspace of a (strong) $\C$-Stegall space.
\item $X$ is a countable union of closed (strong) $\C$-Stegall subspaces of $X$.
\item $X$ is the image of a (strong) $\C$-Stegall space $Z$ under a perfect map $f:Z\to X$;
\item $X$ admits a continuous bijective map $f:X\to Y$ onto a (strong) $\C$-Stegall space $Y$;
\item each non-empty subspace $Y\subset X$ contains a non-empty relatively open subspace $Z\subset Y$ which is (strong) $\C$-Stegall.
\end{enumerate}
\end{proposition}

\begin{proof} The first four statements can be proved by a suitable modification of the proof of Theorem 3.1.5 in \cite{Fab}. To prove the last statement, assume that
each non-empty subspace $Y\subset X$ contains a non-empty relatively open subspace $Z\subset Y$ which is (strong) $\C$-Stegall. To prove that $X$ is (strong) $\C$-Stegall, fix a minimal usco map $\Phi:C\multimap X$ defined on a non-empty space $C\in\C$. We need to show that the set $Z=\{z\in Z:|\Phi(z)|=1\}$ is non-empty (and comeager) in $Z$. This will follow as soon as for every non-empty open set $U\subset Z$ we find a non-empty open set $V\subset U$ such that $V\cap Z$ is non-empty (and comeager) in $V$. By our assumption, the set $\Phi(U)$ contains a non-empty relatively open subspace $W$, which is (strongly) $\C$-Stegall. By Lemma~\ref{l:musco}, the set $U$ contains a non-empty open set $V$ with $\Phi(V)\subset W$. By Lemma~\ref{l:musco}, the restriction $\Phi|V:V\to W\subset X$ is a minimal usco map. Since the  space $W$ is (strong) $\C$-Stegall, the set $\{z\in V:|\Phi(z)|=1\}=V\cap Z$ is non-empty (and comeager) in $V$.
\end{proof}

Since the countable intersection of comeager subsets is comeager, Definition~\ref{d:Stegall} implies the following (almost trivial) proposition.

\begin{proposition}\label{p:cprod} Let $\C$ be a class of Baire spaces. The class of strong $\C$-Stegall spaces is closed under countable products.
\end{proposition}

A lest trivial proposition easily follows from Proposition~\ref{p:Sgood} and Lemma~\ref{l:musco-rest}.

\begin{proposition}\label{p:prods} Let $\C$ be a class of Baire spaces, closed under taking open subspaces and dense $G_\delta$-subsets. For any Hausdorff $\C$-Stegall space $X$ and any strong $\C$-Stegall space $Y$ the product $X\times Y$ is $\C$-Stegall.
\end{proposition}

  We recall that a topological space is called {\em weak Stegall} if it is $\C$-Stegall for the class of completely metrizable spaces. The following example was constructed by Kalenda in \cite{K97}.

\begin{example}[Kalenda]\label{e:K}
There exists a weak Stegall compact Hausdorff space $X$ whose square is not weak Stegall.
\end{example}

\section{Stability properties of the classes of (strong) $\C$-Piotrowski spaces}\label{s:stabP}

In this section we establish some stability properties of the class of (weak) Piotrowski spaces.

\begin{theorem} Let $\C$ be a class of Baire spaces, closed under taking open subspaces.
A topological space $X$ is (strong) $\C$-Piotrowski if one of the following conditions is satisfied:
\begin{enumerate}
\item $X$ is a subspace of a (strong) $\C$-Piotrowski space;
\item $X$ is the countable union of closed (strong) $\C$-Piotrowski subspaces of $X$;
\item each non-empty subspace $Y\subset X$ contains a non-empty relatively open subspace $Z\subset Y$ which is (strong) $\C$-Piotrowski.
\end{enumerate}
\end{theorem}

\begin{proof} 1. The first statement is trivial.
\smallskip

2. Assume that $X$  can be written as the countable union $X=\bigcup_{n\in\w}X_n$  of closed (strong) $\C$-Piotrowski subspaces $X_n$ of $X$.

To show that $X$ is (strongly) $\C$-Piotrowski, take any quasicontinuous function $f:Z\to X$, defined on a non-empty Baire space $Z\in\C$. We shall prove that the set $C(f)$ of continuity points of $f$ is dense (and comeager) in $Z$.
It suffices in every non-empty open set $U\subset Z$ to find a non-empty open set $V\subset U$ such that the intersection $V\cap C(f)=C(f|V)$ is dense (and comeager) in $V$.

Since the space  $U=\bigcup_{n\in\w}U\cap f^{-1}(X_n)$ is Baire, for some $n\in\w$ the set $U\cap f^{-1}(X_n)$ is not meager in $U$. Then there exists a non-empty open set $V\subset U$ such that $V\cap f^{-1}(X_n)$ is dense in $V$. We claim $V\subset f^{-1}(X_n)$. Assuming that $V\not\subset f^{-1}(X_n)$, find a point $v\in V\setminus f^{-1}(X_n)$. Since $f(v)\in X\setminus X_n$, we can use the quasicontinuity of $f$ and find a non-empty open set $W\subset V$ such that $f(W)\subset X\setminus X_n$. Since $V\cap f^{-1}(X_n)$ is dense in $V$, there is a point $z\in W\cap f^{-1}(X_n)$. For this point we get $f(z)\in f(W)\cap X_n\subset (X\setminus X_n)\cap X_n=\emptyset$, which is a desired contradiction witnessing that $V\subset f^{-1}(X_n)$. Since the space $X_n$ is (strong) $\C$-Piotrowski, the set $C(f|V)=C(f)\cap V$ is dense (and comeager) in $V$.
\smallskip

3. Assume that every non-empty subspace $A\subset X$ contains a non-empty relatively open subspace, which is (strong) $\C$-Piotrowski.
 To show that $X$ is (strong) $\C$-Piotrowski, take any quasicontinuous function $f:Z\to X$, defined on a non-empty Baire space $Z\in\C$. We shall prove that the set $C(f)$ of continuity points of $X$ is dense (and comeager) in $Z$.
It suffices in every non-empty open set $U\subset Z$ to find a non-empty open set $V\subset U$ such that the intersection $V\cap C(f)=C(f|V)$ is dense (and comeager) in $V$. By our assumption, the image $f(U)$ contains a non-empty relatively open (strongly) $\C$-Piotrowski subspace $W\subset f(U)$. The quasicontinuity of $f$ yields a non-empty open subset $V\subset U$ such that $f(V)\subset W$. Since the space $W$ is (strong) $\C$-Piotrowski, the set $C(f|V)=V\cap C(f)$ is dense (and comeager) in $V$.
\end{proof}

 \begin{proposition} Assume that a class $\C$ of Baire spaces is closed under taking open subspaces and dense Baire subspaces. A regular topological space $X$ is (strong) $\C$-Piotrowski if $X$ can be written as the countable union $X=\bigcup_{n\in\w}X_n$ of $\C$-Piotrowski subspaces of $X$.
 \end{proposition}

\begin{proof} To show that $X$ is $\C$-Piotrowski, take any
quasicontinuous function $f:Z\to X$, defined on a non-empty Baire space $Z\in\C$. Since the space $Z=\bigcup_{n\in\w}f^{-1}(X_n)$ is Baire, for some $n\in\w$ the set $f^{-1}(X_n)$ is not meager in $Z$. Consequently, there exists a non-empty open set $V\subset Z$ such that the intersection $B=V\cap f^{-1}(X_n)$ is a dense Baire subspace of $V$.
The quasicontinuity of the function $f$ implies the quasicontinuity of the restriction $f|B:B\to X_n$. Taking into account that the class $\C$ is closed under taking open subspaces and dense Baire subspaces, we conclude that $B\in\C$. Since the space $X_n$ is $\C$-Piotrowski, the function $f|B$ has a continuity point $z\in B$. Lemma~\ref{l:Prest} implies that $z$ remains a continuity point of the map $f|V$ and $f$. This witnesses that the space $X$ is $\C$-Piotrowski. By Corollary~\ref{c:P<=>sP}, $X$ is strongly $\C$-Piotrowski.
\end{proof}

The following two propositions are counterparts of Propositions~\ref{p:cprod}, \ref{p:prods}. The first of them follows from Definition~\ref{d:piotr} and the preservation of comeager sets by countable intersections. 

\begin{proposition} Let $\C$ be a class of Baire spaces. The class of strong $\C$-Piotrowski spaces is closed under countable products.
\end{proposition}

The following proposition follows from Definition~\ref{d:piotr} and Proposition~\ref{p:Pgood}(3).

\begin{proposition} Let $\C$ be a class of Baire spaces, closed under taking open subspaces and dense $G_\delta$-sets. For a regular $\C$-Piotrowski space $X$ and a strong $\C$-Piotrowski space $Y$ the product $X\times Y$ is $\C$-Piotrowski.
\end{proposition}

Our last example follows from Theorem~\ref{t:P<=>S} and Example~\ref{e:K}.

\begin{example} There exists a weak Piotrowski compact Hausdorff space $X$ whose square is not weakly Piotrowski.
\end{example}

\end{document}